\documentclass[12pt,reqno]{article}
mm \textheight=8.9in

\usepackage{amssymb}
\usepackage{amsmath}
\usepackage{amsthm}
\usepackage{pb-diagram}
\usepackage{color}
\newcommand{\A}{{\mathcal{A}}}

\newcommand{\br}[3]{{$#1$}$\lower4pt\hbox{$\tp\atop\raise4pt \hbox{$\scriptscriptstyle{#2}$}$} ${$#3$}}
\newcommand{\tw}[3]{{$#1$}${\,\scriptscriptstyle {#2}}\atop\raise9pt\hbox{$\scriptstyle\tp$} ${$#3$}}
\newcommand{\ttps}[2]{{#1}\raise5pt\hbox{$\lower12pt\hbox{$\scriptstyle\tp$}\atop \lower0pt\hbox{$\tilde\;$}$}\raise4.5pt\hbox{${\scriptstyle{#2}}$}}
\newcommand{\st}[1]{\mbox{${\,\scriptscriptstyle {#1}}\atop\raise5.5pt\hbox{$*$}$}}

\newcommand{\rd}[1]{\mbox{${\,\scriptscriptstyle {#1}}\atop\raise5.5pt\hbox{$\bullet$}$}}
\newcommand{\rt}[1]{\otimes_\chi}
\newcommand{\lt}[1]{\mbox{${\,\scriptscriptstyle {#1}}\atop\raise5.5pt\hbox{$\ltimes$}$}}
\newcommand{\btr}{\raise1.2pt\hbox{$\scriptstyle\blacktriangleright$}\hspace{2pt}}
\newcommand{\btl}{\raise1.2pt\hbox{$\scriptstyle\blacktriangleleft$}\hspace{2pt}}

\newcommand{\lcr}{\raise1.0pt \hbox{${\scriptstyle\rightharpoonup}$}}
\newcommand{\rcr}{\raise1.0pt \hbox{${\scriptstyle\leftharpoonup}$}}

\newcommand{\ttp}{{\lower12pt\hbox{$\tp$}\atop \hbox{$\tilde\;$}}}

\newcommand{\id}{\mathrm{id}}

\newcommand{\B}{\mathcal{B}}

\newcommand{\Sg}{\mathfrak{S}}

\newcommand{\Ru}{\mathcal{R}}

\newcommand{\Q}{\mathcal{Q}}

\newcommand{\C}{\mathbb{C}}

\newcommand{\tp}{\otimes}

\newcommand{\tht}{\theta}

\newcommand{\U}{U}

\newcommand{\ve}{\varepsilon}

\newcommand{\dt}{\delta}

\newcommand{\op}{\oplus}
\newcommand{\la}{\lambda}

\newcommand{\srm}{\mathrm{s}}

\newcommand{\End}{\mathrm{End}}

\newcommand{\Rm}{\mathrm{R}}

\newcommand{\g}{\mathfrak{g}}
\renewcommand{\b}{\mathfrak{b}}

\newcommand{\h}{\mathfrak{h}}

\newcommand{\s}{\mathfrak{s}}

\renewcommand{\o}{\mathfrak{o}}

\newcommand{\nn}{\nonumber}
\newcommand{\p}{\mathfrak{p}}
\renewcommand{\l}{\mathfrak{l}}

\newcommand{\si}{\sigma}
\newcommand{\al}{\alpha}

\newcommand{\bt}{\beta}

\newcommand{\be}{\begin{eqnarray}}
\newcommand{\ee}{\end{eqnarray}}

\newtheorem{thm}{Theorem}[section]
\newtheorem{propn}[thm]{Proposition}
\newtheorem{lemma}[thm]{Lemma}
\newtheorem{corollary}[thm]{Corollary}

\newtheorem{example}[thm]{Example}

\newcount\prg

\newcommand{\parag}{\advance\prg by1 {\noindent\bf\thesection.\the\prg\hspace{6pt}}}

\begin{document}
\title{Regularization of Mickelsson generators for non-exceptional quantum groups.}
\author{
Andrey Mudrov\footnote{This research is supported in part by the RFBR grant 15-01-03148.}\vspace{20pt}\\
\small Department of Mathematics,\\ \small University of Leicester, \\
\small University Road,
LE1 7RH Leicester, UK\\
}

\date{}
\maketitle

\begin{abstract}
Let $\g'\subset \g$ be the pair of Lie algebras of either symplectic or orthogonal infinitesimal endomorphisms
of the complex vector spaces $\C^{N-2}\subset \C^N$ and $U_q(\g')\subset U_q(\g)$ the pair of quantum groups
with  triangular decomposition $U_q(\g)=U_q(\g_-)U_q(\g_+)U_q(\h)$.
Let $Z_q(\g,\g')$ be the corresponding step algebra and regard its generators as rational trigonometric functions $\h^*\to U_q(\g_\pm)$.
We describe their regularization such that the resulting generators do not vanish when specialized at any weight.
\end{abstract}

{\small \underline{Mathematics Subject Classifications}:   17B37, 22E47,  81R50.
}

{\small \underline{Key words}: Mickelson algebras, quantum groups, regularization.
}
\date{}
\maketitle
\section{Introduction}
This rather technical paper is devoted to regularization of generators of Mickelsson algebras, regarded as meromorphic functions
on the weight space.
For a general theory of Mickelsson algebras, the reader is referred to \cite{Mick,Zh1,Zh2} (the classical universal
enveloping algebras) and  \cite{Kek,KO} (quantum groups).
Here we are concerned with the special case related to the pair $\g'\subset \g$
 of Lie algebras of orthogonal/symplectic infinitesimal transformations of a fixed pair of vector spaces $\C^{N-2}\subset \C^N$.

Let $\g'=\g'_-\op \h'\op \g'_+$ be the triangular decomposition compatible with a decomposition  $\g=\g_-\op \h\op \g_+$, i.e.
$\g'_\pm \subset \g_\pm$, with $\h'\subset \h$ being the  Cartan subalgebras.
Let $N(\g,\g')$ denote the normalizer of the left ideal $U_q(\g)\g'_+$, i.e. the maximal subalgebra
in $U_q(\g)$ where $U_q(\g)\g'_+$ is a two-sided
ideal. Then the quotient $N(\g,\g')/U_q(\g)\g'_+$ is called step or Mickelsson algebra
and denoted by $Z_q(\g,\g')$.  Its significance comes from the fact that it preserves the subspace of $\g'$-singular vectors in $\g$-modules
(recall that a  vector is called  singular if it generates the trivial representation of $U_q(\g')$).

The  Mickelsson algebra carries a two-sided $U_q(\h)$-action, and is generated by elements $z_0, z_{\pm \al}$
of weights $0$ and, respectively, $\pm \al$ with $\al \in \Rm^+_{\g}-\Rm^+_{\g'}$ (the set of positive roots of $\g$ minus those of $\g'$).
The element  $z_0$ can be taken from $q^{\h\ominus \h'}$ while $z_{\pm \al}$ have representatives in the Borel subalgebra $U_q(\b_\pm)$.
The generators $z_{\pm \al}$ can be expressed through extremal projectors  \cite{AST1,AST2,T} or alternatively as matrix entries
of reduced Shapovalov inverse \cite{M,AM3} (the classical version appeared in \cite{NM,Pei,Mol}). In both cases, they require the rational extension, $\hat U_q(\g)$, of
$U_q(\g)$ over the ring of fractions $\hat U_q(\h)$ of $U_q(\h)$ with respect to a certain multiplicative system.

Regarding $\hat U_q(\g)$ as a free right $\hat U_q(\h)$-module, one can think of $z_{\pm \al}$ as rational trigonometric
$U_q(\g_\pm)$-valued  functions (raising and lowering operators) on $\h^*$.
Of course, they can be made polynomial upon multiplying by the common denominator of the Cartan coefficients.
We call it natural regularization and denote the regularized generators by $\check z_{\pm\al}$.
There arises the question whether $\check z_{\pm \al}(\la)=0$ at some $\la$.
The answer is given in the present paper.
We prove that for special linear and symplectic $\g$, $\check z_{\pm\al}(\la)\not =0$ for all $\la\in \h^*$.
For orthogonal $\g$, there are $\dt_{\pm \al}\in U_q(\h)$ such that $\check z_{\pm\al}$ are
divisible by $\dt_{\pm\al}$ on the right and the quotient $\check z_{\pm\al}\dt_{\pm\al}^{-1}$ does not turn zero at all weights.

Regularization of negative Mickelsson generators is more or less the same as  regularization of singular vectors of $V\tp M_\la$ for
$V$ the "natural" representation of $U_q(\g)$ and $M_\la$ the Verma module of highest weight $\la$. In such a setting, it is a part
of a more general problem when $V$ is an arbitrary finite dimensional $U_q(\g)$-module. The analogous problem was considered
for classical universal enveloping algebras in \cite{STV} and completely solved for $\g=\s\l(3)$.
We study the special case  $V=\C^N$ due to its significance. Remark that quantum groups
lead to new effects which are absent in the classical version: the factors $\dt_{\pm\al}$ turn to $1$ for $\g=\s\o(2n+1)$ in the limit $q\to 1$.
\section{Quantum group preliminaries}
\label{ssecQUEA}
Throughout the paper,  $\g$ is a complex simple Lie algebra of type $B$, $C$ or $D$.
Due to the natural inclusion $\U_q\bigl(\g\l(n)\bigr)\subset \U_q(\g)$, we do not pay special attention to this case
(which was also covered in \cite{M1}, Corollary 9.2, by different arguments).
We fix a Cartan subalgebra $\h\subset \g$ with the non-degenerate symmetric inner product $(.,.)$ on $\h^*$.
By  $\Rm$ we denote the root system of $\g$ with a fixed subsystem of
positive roots $\Rm^+\subset \Rm$ and the basis of simple roots $\Pi^+\subset \Rm^+$.
For every $\la\in \h^*$ we define its image  $h_\la$ under the isomorphism $\h^*\simeq \h$,
that is $(\la,\bt)=\bt(h_\la)$ for all $\bt\in \h^*$.
We denote by $\rho$ the Weyl vector $\frac{1}{2}\sum_{\al\in \Rm^+}\al $.

Suppose that $q\in \C$ is not a root of unity. Denote by $U_q(\g_\pm)$ the $\C$-algebra generated by  $\{e_{\pm\al}\}_{\al\in \Pi^+}$, subject to the q-Serre relations
\be
\sum_{k=0}^{1-a_{\al\bt}}(-1)^k
\left[
\begin{array}{cc}
1-a_{\al\bt} \\
 k
\end{array}
\right]_{q_{\al}}
e_{\pm \al}^{1-a_{\al\bt}-k}
e_{\pm \bt}e_{\pm \al}^{k}
=0
,
\label{Serre}
\ee
where  $a_{\al\bt}=\frac{2(\al,\bt)}{(\al,\al)}$ is the Cartan matrix, $q_{\al}= q^{\frac{(\al,\al)}{2}}$, and
$$
\left[
\begin{array}{cc}
m  \\ k
\end{array}
\right]_{q}
=
\frac{[m]_q!}{[k]_q![m-k]_q!},
\quad
[m]_q!=[1]_q\cdot [2]_q\cdot\ldots \cdot[m]_q.
$$
Here and further on, $[z]_q=\frac{q^z-q^{-z}}{q-q^{-1}}$ whenever $q^{\pm z}$ make sense.

Denote by $U_q(\h)$ the commutative $\C$-algebra generated by $q^{\pm h_\al}$, $\al\in \Pi^+$. The quantum group $U_q(\g)$ is a $\C$-algebra generated by  $U_q(\g_\pm)$ and $U_q(\h)$ subject
to the relations \cite{D}
\be
q^{ h_\al}e_{\pm \bt}q^{-h_\al}=q^{\pm(\al,\bt)} e_{\pm \bt},
\quad
[e_{\al},e_{-\bt}]=\delta_{\al\bt}\frac{q^{h_{\al}}-q^{-h_{\al}}}{q_\al-q^{-1}_\al}.
\label{comm_rel}
\ee
Although $\h$ is not contained in $U_q(\g)$, still it is convenient to keep reference to $\h$.

Fix the comultiplication in $U_q(\g)$ as in \cite{CP}:
\be
&\Delta(e_{\al})=e_{\al}\tp q^{h_{\al}} + 1\tp e_{\al},
\quad
\Delta(e_{-\al})=e_{-\al}\tp 1 + q^{-h_{\al}} \tp e_{-\al},
\nn\\&
\Delta(q^{\pm h_{\al}})=q^{\pm h_{\al}}\tp q^{\pm h_{\al}},
\nn
\ee
for all $\al \in \Pi^+$.

The subalgebras $U_q(\b_\pm)\subset U_q(\g)$ generated by $U_q(\g_\pm)$ over $U_q(\h)$ are quantized universal enveloping algebras of the
Borel subalgebras $\b_\pm=\h+\g_\pm\subset \g$.
The multiplication map implements an isomorphism $U_q(\g_-)\tp U_q(\g_+)\tp U_q(\h)\to U_q(\g)$ of vector spaces,
which also descends to the decomposition $U_q(\b_\pm )=U_q(\g_\pm)U_q(\h)$.

The notation $n$ is reserved for the rank of $\g$.
We enumerate the elements of $\Pi^+$ so that $\g\l(n)$ is a Lie subalgebra in $\g$ with simple roots $\{\al_i\}_{i=1}^{n-1}$,  and $U_q\bigl(\g\l(n)\bigr)$  the corresponding
quantum subgroup in $U_q(\g)$.
Let $\g\l(\srm)\subset \g\l(n)$ be the maximal subalgebra stable under automorphisms of the Dynkin diagram of $\g$.
One has $\srm=n$ for $\g=\s\p(2n), \s\o(2n+1)$ and $\srm=n-1$ for $\g=\s\o(2n)$.

We use the notation $e_i=e_{\al_i}$ and $f_{i}=e_{-\al_i}$ for  $\al_i\in \Pi^+$ in all cases apart
from $i=n$, $\g=\s\o(2n+1)$, when we set $f_n=[\frac{1}{2}]_q e_{-\al_n}$.
This modifies the relation (\ref{comm_rel}) to
$$
[e_{n},f_{n}]=\frac{q^{h_{\al_n}}-q^{-h_{\al_n}}}{q-q^{-1}}.
$$
All other relations stay intact.


\subsection{Natural representation}
In this section we recall the natural representation of $\g$ in the vector space $\C^N$. Let $\{w_i\}_{i=1}^N$
be the standard basis in $\C^N$. We used the notation $i'=N+1-i$ for all integers $i\in I=[1,N]$ corresponding
to the flip of the integer interval $I$ around the center $\frac{N+1}{2}$. To improve readability of formulas, we use special notation $*=\frac{N+1}{2}$.

The natural representation is constructed as follows.  We assign the matrices
$$
\pi(e_{i})=E_{i,i+1}+ E_{i'-1,i'}, \quad \pi(f_{i})= E_{i+1,i}+ E_{i',i'-1}, \quad \pi(h_{\al_i})= E_{ii}-E_{i+1,i+1}+E_{i'-1,i'-1}-E_{i'i'},
$$
for $i=1, \ldots,n-1$. This defines a direct sum of two representations of the subalgebra $U_q(\g\l(n))$.
We extend it to the representation of $U_q(\g)$ as
$$
\pi(e_{n})= E_{n,*}+ E_{n'-1,n'}, \quad \pi(f_{n})=  E_{*,n}+ E_{n',*}, \quad \pi(h_{\al_n})=
E_{nn}-E_{n'n'},
$$
$$
\pi(e_{n})= E_{nn'}, \quad \pi(f_{n})= E_{n'n}, \quad \pi(h_{\al_n})=
2E_{nn}-2E_{n'n'},
$$
$$
\pi(e_{n})=\! E_{n-1,n'}+ E_{n,n'+1}, \> \pi(f_{n})= \!E_{n',n-1}+ E_{n'+1,n}, \>
\pi(h_{\al_n})=\!
E_{n-1,n-1}+E_{nn}-E_{n'n'}-E_{n'+1,n'+1},
$$
respectively, for $\g=\s\o(2n+1)$, $\g=\s\p(2n)$, and $\g=\s\o(2n)$.
The Cartan subalgebra is represented by diagonal matrices, and the basis elements  $w_i$
carry weights $\ve_i\in \h^*$ with  $\ve_{i'}=-\ve_i$. The set  $\{\ve_i\}_{i=1}^n$ forms an  orthonormal basis  $\h^*$.

We introduce a partial ordering on the integer interval $[1,N]$ by setting $i\preccurlyeq j$ if and only if $w_j\in U_q(\g_-)w_i$.
One has $i\prec j\Rightarrow i<j$.

Define a matrix $F\in \End(V)\tp U_q(\g_-)$ as $F=(\pi\tp \id)(q^{-\sum_{i=1}^nh_{\ve_i}\tp h_{\ve_i}}\Ru)$.
Its entries $f_{ij}$ are expressed through modified commutators $[x,y]_a=xy-ayx$, $a\in \C$,  as given below.
For all $\g$ and  $i<j\leqslant *$ set
\be
f_{ij}=[f_{j-1},\ldots [f_{i+1},f_i]_{\bar q^{}}\ldots ]_{\bar q},
\quad
f_{j'i'}=[\ldots [f_i,f_{i+1}]_{\bar q^{}},\ldots f_{j-1}]_{\bar q}.
\label{a}
\ee
where $f_{i,i+1}=f_i=f_{i'-1,i}$ is understood. Here and further on the bar designates the inverse, e.g. $\bar q=q^{-1}$.
Furthermore,
\begin{itemize}
\item
for $\g=\s\o(2n+1)$: $f_{nn'}=(q-1)f_n^2$ and
\be
f_{ij'}=q^{\dt_{ij}}[f_{*,j'},f_{i,*}]_{\bar q^{\dt_{ij}}},  \quad i,j<n.
\label{b_2}
\ee
\item
for $\g=\s\p(2n)$:  $f_{nn'}=[2]_qf_n$ and
\be
\label{c_1}
f_{in'}=[f_{n},f_{in}]_{\bar q^{2}}, \quad f_{ni'}=[f_{n'i'},f_{n}]_{\bar q^{2}}, \quad i<n,
\ee
\be
\label{c_2}
f_{ij'}=q^{\dt_{ij}}[f_{nj'},f_{in}]_{\bar q^{1+\dt_{ij}}}, \quad i,j<n.
\ee
\item
for $\g=\s\o(2n)$: $f_{nn'}=0$ and
\be
\label{d_1}
f_{in'}=[f_{n},f_{i,n-1}]_{\bar q^{}}, \quad f_{ni'}=[f_{n'+1,i'},f_n]_{\bar q^{}},\quad i<n-2,
\ee
\be
\label{d_2}
f_{ji'}=q^{\dt_{ij}}[f_{ni'},f_{j,n}]_{\bar q^{1+\dt_{ij}}}, \quad i,j< n.
\ee
\end{itemize}
Finally, $f_{ii}=1$ for all $i$ and $f_{ij}=0$ for $i>j$.

The matrix $F$ participates in  construction of reduced Shapovalov inverse form
$\hat F=\sum_{i,j=1}^N E_{ij}\tp \hat f_{ij}$, which is given next \cite{M}.
It is convenient to use the language of Hasse diagram of the $\prec$ -poset $I$, whose arcs are labeled with negative Chevalley generators
(directed toward superior nodes). We call any ascending sequence of
nodes $(m_i)_{i=1}^k\subset I$ a route from $m_1$ to $m_k$.  A maximal route, i.e. whose adjacent nodes are connected with arcs, is called  path.  If $i\prec m_1$ and $m_k\prec j $ for some $i,j$, we write $i\prec \vec m\prec j$.

For all $i, j\in I$ define $\eta_{ij}\in \h+\C$ by
$$
\eta_{ij}=h_{\ve_i}-h_{\ve_j}+\rho_i-\rho_j-\frac{1}{2}||\ve_i-\ve_j||^2,
$$
where $\rho_i=(\rho,\ve_i)$, and $||\bt ||^2$ is the Eucleadean norm of $\bt\in\h^*$. We
regard $\eta_{ij}$ as an affine function $\h^*\to \C$, $\la\mapsto (\la+\rho, \ve_i-\ve_j)-\frac{1}{2}||\ve_i-\ve_j||^2$.
The entries $\hat f_{ij}$  are constructed as follows.
Put $\hat f_{ii}=1$ and $\hat f_{ij}=0$  for $i>j$. For $i< j$, define $A^j_i=\frac{q-q^{-1}}{q^{-2\eta_{ij}}-1}$.
For a route $\vec m$ denote $f_{\vec m}=f_{m_1,m_2}\ldots f_{m_{k-1},m_k}$ and $A^j_{\vec m}=A^j_{m_1}\ldots A^j_{m_{k}}$. Then
$$
\hat f_{ij}=\sum_{i\prec \vec m\prec j}^{\varnothing } f_{i,\vec m, j}A^j_{i,\vec m},
$$
where the symbol $\varnothing$ indicates here and further on that the empty route $\vec m=\varnothing$ is included.
The elements $\hat f_{1j}$, where $j$ ranges from $2$ to $N$ for $\g=\s\p(N)$ and to $N-1$ for $\g=\s\o(N)$
form the set of negative generators of $Z_q(\g,\g')$ where $\g'\subset \g$ is the simple Lie subalgebra with the root basis
$\{\al_2,\ldots, \al_n\}$.

Let $M_\la$ be the Verma module  of highest weight $\la \in \h^*$ with the canonical generator $v_\la$.
The matrix $\hat F$ is regarded as a map $\h^*\mapsto \End(V)\tp U_q(\g_-)$, such that $\hat f_{ij}(\la)v_\la=\hat f_{ij}v_\la$.
The tensors $\hat F_j=\hat F(w_j\tp v_\la)=\sum_{i=1}^j w_i\tp \hat f_{ij}v_\la$ are singular vectors, i.e. annihilated by all $e_\al$, $\al\in \Pi^+$.
They are well defined for generic $\la$  and generate submodules $M_j\simeq M_{\la+\ve_j}\subset V\tp M_\la$.
At some weights, $\hat F_j$ have zeros and poles and therefore need a regularization, as singular vectors
are defined up to a scalar multiplier. In particular, there is a natural regularization
$\check f_{ij}=\hat f_{ij}\prod_{l\prec j}\bar A_l\in U_q(\b_-)$. It turns out to be excessive in some cases as
having zeros at some $\la$. We study this issue for the most important pairs $(i,j)$ relative to the generators of $Z_q(\g,\g')$.


\section{Standard filtration in  $V\tp M_\la$}
Our study of the matrix $\hat F$ is based on analysis of the tensor product $V\tp M_\la$.
To a large extent, its module structure is captured by a $U_q(\g)$-invariant operator $\Q=(\pi\tp \id)(\Ru_{21}\Ru)\in \End(V)\tp M_\la$, which  is scalar on highest weight submodules
and factor modules.
Denoting by $x_j$ its eigenvalue on the submodule $M_j$, one has  $x_ix^{-1}_j=q^{2\xi_{ij}}|_\la$ with
$$
\xi_{ij}=h_{\ve_i}-h_{\ve_j}+\rho_i-\rho_j+\frac{1}{2}(||\ve_i||^2-||\ve_j||^2).
$$
For generic $\la$ the eigenvalues $x_j$ are pairwise distinct and separate $M_{j}$.

Another tool for the analysis of $V\tp M_\la$ is the sequence of submodules $(V_j)_{j=1}^N$ generated by $v_{\la,l}=w_l\tp v_\la$, $l=1,\ldots, j$.
It forms an ascending  filtration of $V\tp M_\la$ whose graded module is isomorphic to the direct sum $\op_{j=1}^N V_j/V_{j-1}$ of Verma modules.
This implies that $V_j$ are all $\Q$-invariant, \cite{AM2}.
In the present section we study projection  $\wp_j\colon  M_{j} \to V_j/V_{j-1}$, which can be either
zero or an isomorphism. By weight arguments, $M_{j}\subset V_j$, and
the singular vector $\hat F_j$ is mapped into the line generated $v_{\la,j} \mod V_{j-1}$, the highest vector of $V_j/V_{j-1}$.
Define
$\hat C_i$  by $\wp_j(\hat F_j)=\hat C_jv_{\la,j}$.
Our nearest goal is computation of $\hat C_j$, $j=1,\ldots, N$. It is clear that
$\hat C_1=1$. For $j>1$ the answer is given by Proposition \ref{quatnum_C-A} below.

For all $j\in I$ introduce a commutative algebra
$\A_{j}$ as follows. For $j\leqslant \frac{N+1}{2}$ set $\A_j=\C[y_1^{\pm 1},\ldots,y_{j-1}^{\pm 1}]$.
Otherwise put $\A_j$ to be the quotient of $\C[y_1^{\pm 1},\ldots,y_{j-1}^{\pm 1}]$ modulo the relations
$y_ly_{l'}=y_{j'}$, $l=j'+1,\ldots, j-1$, and, for $\g=\s\o(2n+1)$, extended with $y_{j'}^{\frac{1}{2}}$ subject to $qy_*=y_{j'}^{\frac{1}{2}}$.
In all cases,  $\A_{j}$ is a localization of a polynomial algebra.
We can realize $\A_j$ as a subalgebra in $\hat U_q(\h)$ via the assignment $y_l=q^{-2\eta_{lj}}$, $l\prec j$, due to
the following fact.
\begin{propn}
For all $m=2,\ldots,n$ one has $\eta_{m1'}+\eta_{m'1'}=\eta_{11'}$.
Also, for $\g=\s\o(2n+1)$, $2\eta_{*,m'}-1=\eta_{m,m'}$, $m=1,\ldots,n$.
\end{propn}
\begin{proof}
Straightforward.
\end{proof}
\noindent
We regard $\hat f_{ij}$ as an element of the free $\A_j$-module
generated by $U_q(\g_-)$.

Consider $\hat C_j$ as a polynomial in $B_i=-A_i$, $i=1,\ldots, j-1$,
where  $A_i=\frac{q-\bar q}{y_i-1}$. Let $|\!\!|i-j|\!\!|$ denote the number of arcs in a path from $i$ to $j$ on the Hasse diagram
\begin{propn}
The coefficients $\hat C_{j}$, $j=2,\ldots,N$, factorize as
$$
\hat C_{j}=
\left\{
\begin{array}{rrrll}
(1-[2]_qq^2A_j){\displaystyle \prod_{i=1,\atop i\not =j'}^{j-1}}(1- qA_i)&\g=\s\p(N),
\\[15pt]
\frac{y_*-q}{y_*-\bar q}{\displaystyle \prod_{i=1,\atop i\not = j', *}^{j-1}}(1- qA_i)
, &\g=\s\o(2n+1),
\\[15pt]
{\displaystyle \prod_{i=1,\atop i\not = j'}^{j-1}}(1- qA_i)
, &\g=\s\o(2n),
\end{array}
\right.
$$
where factor $\frac{y_*-q}{y_*-\bar q}$  is present only when $*<j$.
\label{quatnum_C-A}
\end{propn}

The proof is based on the concept of principal monomial, which is associated with every pair $i\prec j$.
Each path from $i$ to $j$ gives rise to a unique element  $\psi_{ji} \in U_q(\g_-)$ such that  $w_j=\psi_{ji} w_i$.
We denote by  $\psi^{ij}$ the element obtained from $\psi_{ji}$ by reverting the order of simple factors and call it principal monomial
of the pair $(i,j)$.
\begin{lemma}[\cite{AM2}]
Suppose that $i\prec j$ and $\psi \in U_q(\g_-)$ is a Chevalley monomial of weight $\ve_j-\ve_i$ distinct from $\psi^{ij}$.
Then $\wp_{j}(\psi)=0$. Furthermore
$\wp_{j}(w_i\tp \psi^{ij}v_\la)=(-1)^{|\!\!|i-j|\!\!|}q^{\tilde \rho_i-\tilde \rho_j}v_{\la,j}$,
where $\tilde \rho_i=\rho_i+\frac{1}{2}||\ve_i||^2$.
\label{principal}
\end{lemma}
\begin{proof}
The first part of the statement is proved in \cite{AM2}, Lemma 3.4. Similar statement for a different version of the
comultiplication  is also proved therein. Here we give a proof for the current version of the quantum group.
Suppose that $\al\in \Pi^+$ and $\ve_i-\ve_k=\al$. By \cite{AM2}, Lemma 3.4, the node $w_{i}\tp \psi^{kj}v_\la$ lies
in $V_{j-1}$. Applying $\Delta f_\al= f_\al\tp 1+q^{-h_\al}\tp f_\al$ to $w_{i}\tp \psi^{kj}v_\la$ we get
\be
w_i\tp \psi^{ij}v_\la &=&-q^{(\al,\ve_i)}w_{k}\tp \psi^{kj}v_\la
=-q^{\tilde \rho_i-\tilde \rho_k}w_{k}\tp \psi^{kj}v_\la \mod  V_{j-1}
\ee
for all $k\preccurlyeq j$. Here we used $f_\al w_i=w_k$ and $f_\al\psi^{kj}=\psi^{ij}$ for all $k\preccurlyeq j$.
Proceeding along the path from $i$ to $j$ we complete the proof.
\end{proof}
Thanks to Lemma \ref{principal}, only the principal term of $\hat f_{ij}$ contributes to $\wp_{j}(w_{j}\tp v_\la)$.
Put $\si=1$ for symplectic $\g$ and $\si=-1$ for orthogonal $\g$.
\begin{lemma}
\label{expl_f}
The $(i,j)$-principal term of  $f_{ij}$ is $(-1)^{|\!\!|i-j|\!\!|-1} c_{ij}\psi^{ij}$ with
$$
c_{ij}=
\left\{\begin{array}{rcl}
 \bar q^{\eta_{ij}(0)},& i\not = j',&\\
\bar q \bar q^{\eta_{ij}(0)}+\si q,&i = j'.& \\
\end{array}
\right .
$$
\end{lemma}
\begin{proof}
Straightforward.
\end{proof}
Given a route $\vec m =(m_1,\ldots,m_k)$ put
$c_{\vec m}=c_{m_1,m_2}\ldots c_{m_{k-1},m_k}$.
Introduce $\hat c_{ij}\in \C$ via the equality $\wp_j(w_i\tp \hat f_{ij}v_\la)=\hat c_{ij}v_{\la,j}$, so that $\hat C_j=\sum_{i=1}^{j}\hat c_{ij}$.
It is easy to  check $\hat c_{ij}=\sum_{i\prec \vec m\prec j}^{\varnothing }c_{i,\vec m,j}B_{i,\vec m}q^{\tilde  \rho_i-\tilde \rho_j}$.

\begin{proof}[Proof of Proposition \ref{quatnum_C-A} for $j=N$]

We prove this special case by induction on $N$. The base for induction
is immediate $\hat C_2=1$ for $\g=\s\o(2)$. It is less obvious although straightforward
$\hat C_3
=\frac{y_*-q}{y_*-\bar q}$ for $\g=\s\o(3)$
and
$\hat C_{2}=1-[2]_qq^{2}A_1$ for $\g=\s\p(2)$.
So we assume $\ve_2\not =0$ for higher $N$ in what follows.
Our strategy is to factor out
$
\phi_{i}=B_{i}+q^{-1}
$
for $i=2',2$. This will facilitate the induction transition.

With $1+\hat c_{2'1'}=1+B_{2'}q^{\tilde\rho_{2'}-\tilde\rho_{1'}}=q\phi_{2'}$, let us
 calculate $\sum_{i=3}^{2'}\hat c_{i1'}$ next. Observe that for all $3\prec\vec m \prec 2'$ we have the equality
$c_{i,\vec m,1'}=q^{-1}c_{i,\vec m,2',1'}=q^{-1}c_{i,\vec m,2'}$. Also, for all $i$ we replace $q^{\tilde\rho_i-\tilde\rho_{1'}}=qq^{\tilde\rho_i-\tilde\rho_{2'}}$. Then
\be
1+\sum_{i=3}^{2'}\hat c_{i1'}&=&q\phi_{2'}+ \sum_{i=3}^{3'}\sum_{i\prec\vec m\prec 2'}^\varnothing c_{i,\vec m,2',1'}B_{i,\vec m}B_{2'}q^{\tilde\rho_i-\tilde\rho_{1'}}+\sum_{i=3}^{3'}\sum_{i\prec\vec m\prec 2'}^\varnothing c_{i,\vec m,1'}B_{i,\vec m}q^{\tilde\rho_i-\tilde\rho_{1'}}.
\nn\\
&=&
q\phi_{2'}(1+ \sum_{i=3}^{3'}\sum_{i\prec\vec m\prec 2'}^\varnothing c_{i,\vec m,2'}B_{i,\vec m}q^{\tilde\rho_i-\tilde\rho_{2'}}).
\nn
\ee
The sum $\sum_{i=1}^2\hat c_{i1'}$  reads
\be
&=&\sum_{i=1}^2\sum_{i\prec \vec m\prec 1'}^\varnothing c_{i,\vec m,1'}B_{i,\vec m}q^{\tilde\rho_i-\tilde\rho_{1'}}=
\sum_{i=1}^2\sum_{i\prec \vec m\prec 2'}^\varnothing \bigl(c_{i,\vec m,2',1'}B_{i,\vec m,2'}q^{\tilde\rho_i-\tilde\rho_{1'}}+c_{i,\vec m,1'}B_{i,\vec m}q^{\tilde\rho_i-\tilde\rho_{1'}}\bigr)
\nn\\
&=&
\phi_{2'}\sum_{i=1}^2\sum_{i\prec \vec m\prec 2'}^\varnothing c_{i,\vec m,2'}B_{i,\vec m}q^{\tilde\rho_i-\tilde\rho_{1'}}+
\sum_{i=1}^2\sum_{i\prec \vec m\prec 2'}^\varnothing (c_{i,\vec m,1'}-q^{-1}c_{i,\vec m,2'})B_{i,\vec m}q^{\tilde\rho_i-\tilde\rho_{1'}}.
\nn
\ee
The only nonzero differences in the last sum correspond to $i=1, \vec m=\varnothing, \vec m=(2)$
and $i=2, \vec m=\varnothing$.
They are equal to $c_{1,1'}-q^{-1}c_{1,2'}=q^{-\tht_{1,1'}-1}+\si q-q^{-1}q^{-\tht_{1,2'}}=\si q$
and $c_{1,2, 1'}-q^{-1}c_{1,2,2'}=c_{2,1'}-q^{-1}c_{2,2'}=q^{-\tht_{2,1'}}-\si - q^{-1}q^{-\tht_{2,2'}-1} =-\si$, respectively.
This gives the last term
$$
\si q\bigl( qB_{1}- B_{1,2} - B_{2}q^{\tilde\rho_2-\tilde\rho_{1}}\bigr)q^{\tilde\rho_1-\tilde\rho_{2'}}
=\si q\phi_{2'} \bigl(B_1-B_2\bigr)q^{\tilde\rho_1-\tilde\rho_{2'}},
$$
since $\tilde\rho_2-\tilde\rho_{1}=\rho_2-\rho_{1}=-1$.
This way we factor out $q\phi_{2'}$ in the expression for $\hat C_{1'}$:
\be
\frac{\hat C_{1'}}{q\phi_{2'}}&=&
1+ \sum_{i=1}^{3'}\sum_{i\prec \vec m\prec 2'}^\varnothing c_{i,\vec m,2'}B_{i,\vec m}q^{\tilde \rho_i-\tilde \rho_{2'}}
+\si  \bigl(B_1-B_2\bigr)q^{\tilde\rho_1-\tilde\rho_{2'}}.
\label{C/F2'}
\ee
Consider separately the sums over $i=1,2$ and $3\leqslant i \leqslant 3'$.
First, using the equality
$
c_{12'}-q^{-1}c_{22'}=
q^{-\tht_{1,2'}}-\si - q^{-\tht_{2,2'}-2} =-\si
$,
develop the  internal sum for $i=1$ as
$$
\sum_{2\prec \vec m\prec 2'}^\varnothing \bigl(c_{1,2,\vec m,2'}B_{1,\vec m}B_2+c_{1,\vec m,2'}B_{1,\vec m}\bigr)p^{\tilde \rho_1-\tilde \rho_{2'}}
=
q\phi_2\sum_{2\prec \vec m\prec 2'}^\varnothing c_{2,\vec m,2'}B_{1,\vec m}q^{\tilde \rho_2-\tilde \rho_{2'}}-\si B_{1}q^{\tilde \rho_1-\tilde \rho_{2'}},
$$
where the substitution $q^{\tilde \rho_1-\tilde \rho_{2'}}=qq^{\tilde \rho_2-\tilde \rho_{2'}}$ is made.

Next observe that for  $3\leqslant  k \leqslant 3'$ and  $l=1,2$, we have
$
c_{lk}=q^{-\tht_{lk}}=q^{-\rho_l+\rho_k+\frac{||\ve_l-\ve_k||^2}{2}}=q^{-\tilde \rho_l+\tilde\rho_k+\frac{||\ve_l-\ve_k||^2}{2}+
\frac{||\ve_l||^2-||\ve_k||^2}{2}}=q^{-\tilde \rho_l+\tilde\rho_k+1},
$
since $\ve_2\not =0$ by the assumption. Rewrite the internal sum for $i=2$ in (\ref{C/F2'}) as
$$
c_{22'}B_{2}q^{\tilde \rho_2-\tilde \rho_{2'}}+\sum_{k=3}^{3'}\sum_{k\prec \vec m\prec 2'}^\varnothing c_{k,\vec m,2'}B_{k,\vec m}B_{2}q^{\tilde \rho_k-\tilde \rho_{2'}+1}.
$$
Along with the sum over $3\leqslant i\leqslant 3'$ in (\ref{C/F2'}), this  gives
$$
=-c_{22'}q^{\tilde \rho_2-\tilde \rho_{2'}}+\phi_{2}c_{22'}q^{\tilde \rho_2-\tilde \rho_{2'}}+\phi_2\sum_{i=3}^{3'}\sum_{i\prec \vec m\prec 2'}^\varnothing c_{i,\vec m,2'}B_{i,\vec m}q^{\tilde \rho_i-\tilde \rho_{2'}+1}.
$$
Upon these transformations, (\ref{C/F2'}) turns into
\be
\frac{\hat C_{1'}}{q\phi_{2'}}&=&
1-c_{22'}q^{\tilde \rho_2-\tilde \rho_{2'}-1}
 -\si (\phi_2-q^{-1}) q^{\tilde \rho_1-\tilde \rho_{2'}}  +\phi_2c_{22'}q^{\tilde \rho_2-\tilde \rho_{2'}}
\nn\\
& +& q\phi_2\sum_{i=3}^{3'}\sum_{i\prec \vec m\prec 2'}^\varnothing c_{i,\vec m,2'}B_{i,\vec m}q^{\tilde \rho_i-\tilde \rho_{2'}}
+q\phi_2\sum_{2\prec \vec m\prec 2'}^\varnothing c_{2,\vec m,2'}B_{1,\vec m}q^{\tilde \rho_2-\tilde \rho_{2'}}.
\nn
\ee
Since
$
 -\si q^{\tilde \rho_1-\tilde \rho_{2'}}  +c_{22'}q^{\tilde \rho_2-\tilde \rho_{2'}}
= -\si q^{\tilde \rho_1-\tilde \rho_{2'}}  +(q^{-\tht_{22'}-1}+\si q)q^{\tilde \rho_2-\tilde \rho_{2'}}
= q^{-\tht_{22'}-1}q^{\tilde \rho_2-\tilde \rho_{2'}}=q
$,
the first line gives $q\phi_{2}$.
 Eventually,
we get
$$
\hat C_{1'}/\phi_{2}\phi_{2'}q^2=
1  + \sum_{i=2}^{3'}\sum_{i\prec \vec m\prec 2'}^\varnothing c_{i,\vec m,2'}B_{i,\vec m}q^{\tilde \rho_i - \tilde \rho_{2'} }.
$$
The right-hand side is exactly $\hat C_{2'}(y_{1}, y_{3},\ldots, y_{3'})$ for $\dim V=N-2$. This yields $\hat C_{1'}$ by induction on $N$.
\end{proof}
\begin{proof}[Proof of Proposition \ref{quatnum_C-A} for the general $j$]
Now we assume  $1<j'\leqslant j$.
Let $k$ be $j+1$, so that $k'=j'-1$.
Using the formula $c_{1i}=q^{\tilde \rho_i-\tilde \rho_1+1}$ for $1<i<j'-1$, present
$\hat c_{1j}=\sum_{1\prec \vec m\prec j'}^\varnothing c_{1,\vec m}B_{1,\vec m}q^{\tilde \rho_1-\tilde \rho_j}$ as
$$
c_{1j}q^{\tilde \rho_1-\tilde \rho_{j}}+B_{1}\sum_{i=2}^{j-1}\sum_{i\prec \vec m\prec j}^\varnothing c_{1i}c_{i,\vec m}B_{i,\vec m}q^{\tilde \rho_1-\tilde \rho_{j}}
=qB_{1}(1+\sum_{i=2}^{j-1}\sum_{i\prec \vec m\prec j'}^\varnothing c_{i,\vec m}B_{i,\vec m}q^{\tilde \rho_i-\tilde \rho_{j}}).
$$
This is equal to $qB_1(1+\sum_{i=2}^{j-1} \hat c_{ij})$.
$\hat C_{j}=1+\sum_{i=2}^{j-1} \hat c_{ij}+ \hat c_{1j}=qB_1(1+\sum_{i=2}^{j-1} \hat c_{ij})$.
Induction gives $\hat C_{j}=qB_1\ldots qB_{j'-1}(1+\sum_{i=j'}^{j-1} \hat c_{ij})$.
The expression in the bracket is nothing but $\hat C_j(y_{j'},y_{j'+1},\ldots, y_{j-1})$ for $\dim V=j$. Its factorization is already proved.
\end{proof}
We apply Proposition \ref{quatnum_C-A} for the analysis of the matrix entries $\check f_{ij}$ specifically for orthogonal $\g$.
Introduce $\dt_j^-\in \A_j$ as follows. Put $\dt^-_j=1$ for $j<\srm'$ and, for $j\geqslant \srm'$,  $\dt_j^-=\frac{y_{j'}-1}{q-\bar q}=\bar A_{j'}$ if $\s\o(2n)$ and $\dt_j^-=\frac{y_{j'}^{\frac{1}{2}}+1}{q+1}$ if $\s\o(2n+1)$. For symplectic $\g$, set $\dt_j^-=1$ for all $j$.

\begin{lemma}
Suppose that $\g=\s\o(N)$. For all $i,j$ such that and $i\leqslant j'$ the element  $\check f_{ij}$
is  divisible by $\dt_j^-$.
\label{regularitB_so}
\end{lemma}
\begin{proof}
Only the case  $i\leqslant \srm$,  $\srm'\leqslant j$ requires consideration.
Without loss of generality, we can set  $i=1$.
Recall the basis decomposition $\hat F_{j}=\sum_{l=1}^{j}w_l\tp \hat f_{lj}v_\la$
and consider the singular vector
$\hat F_{j}^\sharp =\frac{y_{j'}-1}{q-\bar q}\hat F_{j}$.
Under the specialization $q^{-2\eta_{l1'}}|_\la=y_l$, $l\prec j$, we have $\hat f_{1j}(\la)=\hat f^\sharp_{1j}\frac{q-\bar q}{y_{j'}-1}$.
Observe that $y_{j'}-1$ is divisible by $\dt_j^-$, and $A_{j'}=\frac{q-\bar q}{y_{j'}-1}$ is the only
factor in $\hat f_{1j}$ that has a pole at $\dt_j^-=0$, so $\hat F_{j}^\sharp$ is regular at $\dt_j^-=0$.
Since $\hat C_{j}$ is regular at generic $\la$ subject  to
 $\dt_j^-=0$, one has $\wp_{j}(\hat F_{j}^\sharp)=0$.
This is possible only if either $\hat F_{j}^\sharp=0$ or $\hat F_{j}^\sharp\in V_{j-1}$. Of all $\Q$-eigenvalues
 only $x_j,x_{j'}$ for even $N$ and $x_j,x_*,x_{j'}$ for odd $N$ become constant at $\dt_j^-=0$,
 while the other can be made distinct from $x_j$.
Since $x_{j}x_{j'}^{-1}= q^{-4}y_{j'}\not =1$
and
$
x_{j}x_{*}^{-1}=q^{-1}y_{j'}^{\frac{1}{2}}\not=1
$,
the spectrum of $\Q$ restricted to $V_{j-1}$ does not contain $x_{j}$, its eigenvalue on $M_{j}$.
Therefore $\hat F_{j}^\sharp$ vanishes along with its coefficient $\hat f^\sharp_{1j}$,  subject to $\dt_j^-=0$.
Hence  $\check f_{1j}$ is divisible by $\dt_j^-$.
\end{proof}

\section{Regularization of matrix coefficients}
In this section we describe regularization for $\hat f_{ij}$.
 We consider reductions of $\check f_{ij}$ along  two-sided ideals in  $\B=U_q(\g_-)$ from the following family.
\label{B_k}
Denote by $\Sg$ the union of integer intervals $[1,\srm]\cup [\srm',1']$ and by $\Sg_{ij}=]i,j[\cap \Sg$ for all pairs $i\prec j$. Fix $k\in \Sg$ and let $l,r$  be the nearest nodes to $k$ such that $l\prec k\prec r$. Suppose that
$f_\al$ and $f_\bt$ corresponding to $\al=\ve_k-\ve_l$ and $\bt=\ve_r-\ve_k$ do not commute. We define
the two sided ideal  $J_{k}\subset \B$ generated by all such $f_{lr}$ (exactly one unless $\g=\s\o(2n)$ and $k=\srm, \srm'$).
As follows from (\ref{Serre}),  $\B/J_{k}$ is isomorphic as a vector space to $U_q(\g^1_-)\tp U_q(\g^2_-)$, where $\g^i\subset \g$ are certain Lie subalgebras. One has the relation $ f_\bt f_{\al}=\bar q f_{\al}f_\bt$ in all cases apart from $\g=\s\p(2n)$  and $k=n,n'$; then
$f_\bt f_{\al}=\bar q^2 f_{\al}f_\bt$. Furthermore
\begin{lemma}
For all $i\prec j$ and $k \in \Sg_{ij}$, one has $f_{ij}=0$ modulo $J_{k}$ and $f_{ij}=(q-\bar q)f_{ik}f_{kj}$ modulo $J_{k'}$.
\end{lemma}
\begin{proof}
Readily follows from (\ref{a}-\ref{d_2}).
\end{proof}
\begin{corollary}
\label{factorization}
Suppose that  $i\prec j$. Then
\begin{itemize}
  \item [i)]
  for all $k \in \Sg_{ij}$,
 $\check f_{ij}=\check f_{ik}\check f_{kj}\mod J_k$ and $\check f_{ij}=\check f_{ik}\check f_{kj}y_k\mod J_{k'}$,
  \item [ii)]
if $k, k' \in \Sg_{ij}$ and $k<k'$,
then
$\check f_{ij}=\check f_{ik}\check f_{kk'}\check f_{k'j}y_{k'}\mod J_k$.

\end{itemize}
\end{corollary}
\begin{proof}
Let us check i) first. Fix  $k\in \Sg_{ij}$ and write
$$
\hat f_{ij}=\sum_{i\prec \vec m\prec k \prec \vec l\prec j}^\varnothing
(f_{i,\vec m, k}A_{i,\vec m}f_{k,\vec l ,j}A_{\vec l}A_{k}+
f_{i,\vec m,\vec l ,j}A_{i,\vec m}A_{\vec l}).
$$
The second sum vanishes modulo $J_k$, and this implies the first factorization.
The element $f_{i,\vec m,\vec l ,j}$  factorizes as $(q-q^{-1}) f_{i,\vec m, k}f_{k,\vec l ,j}$ modulo $J_{k'}$ in accordance with Lemma \ref{B_k}.
Taking into account that $(A_{k}+q-\bar q)\bar A_{k}=y_{k}$, we get factorization modulo $J_{k'}$.
Finally, observe that ii) follows from i).
\end{proof}
\begin{corollary}
\label{not 0 partial}
The elements $\check f_{ij}(\la)$ with  $i\prec j< \srm'$ or $\srm< i\prec j$ do not turn zero at all $\la$.
\end{corollary}
\begin{proof}
Indeed,  projecting $\check f_{ij}$  along $\sum_{k\in \Sg_{ij}}J_k$ we get $f_{\vec m}$,
where $\vec m$ is the (unique) path from $i$ to $j$. All factors in the product are Shevalley generators,
therefore  $f_{\vec m}$ and hence $\check f_{ij}$ do not vanish.
\end{proof}

\begin{thm}
\label{regularization}
Suppose that  $i\prec j$ and $i\leqslant j'$.
Then $\check f_{ij}(\la)/\dt^-_j(\la)\not =0$ at all $\la\in \h^*$.
\end{thm}
\begin{proof}
  For $\g=\s\p(2n)$ we have $\srm=n$ and $\dt^-_j=1$. In view of Lemma \ref{not 0 partial}, it is sufficient to check the case $i\leqslant n$, $n'\leqslant j$.
Taking projection modulo $J_\srm$ we get factorization $\check f_{ij}=\check f_{in}\check f_{nn'}\check f_{n'j}y_{n'}$. It is not zero
as the left and right factors in the product are non-zero by Corollary \ref{not 0 partial}, and $f_{nn'}=[2]_qf_n\not =0$.
Note that restriction $i\leqslant j'$ can be relaxed for symplectic $\g$.

The case of  orthogonal $\g$ and  $j< \srm'$  is covered by Corollary \ref{not 0 partial}, as $\dt^-_j=1$ then.

Now suppose that  $\g$ is orthogonal and  $s'\leqslant j$ (implying  $i\leqslant s$).
Projection modulo $J_\srm$ yields
$\hat f_{ij}/\dt_j^-=\hat f_{in}(\hat f_{\srm \srm'}/\dt_j^-)\hat f_{n'j}$. It is not zero provided $\hat f_{\srm \srm'}/\dt_j^-\not =0$.
We find it equal to $f_n^2$ in for odd $N$ and $f_{n-1}f_n$ for even $N$. This completes the proof.
\end{proof}
This theorem gives regularization of the negative  Mickelsson generators, $\hat z_{-\al_j}=\hat f_{1j}$,
for $\al=\ve_1-\ve_j\in \Rm^+$.
\begin{example}
\em
Let us illustrate Theorem \ref{regularization} on the example of $\check f_{15}$ for $\g=\s\o(6)$.
The algebra $\A_5$ is generated by $y_1,\ldots, y_4$ subject to $y_3y_4=y_2$. Then  $d_5^-=\bar A_2=\frac{y_2-1}{q-\bar q}$. With $f_{34}=0$,
$\hat f_{15}$  reads
$$
\hat f_{15}=f_{15}A_1+f_{13}f_{35}A_{1,3}+f_{14}f_{45}A_{1,4}
+f_{12}f_{25}A_{1,2}+f_{12}f_{23}f_{35}A_{1,2,3}+f_{12}f_{24}f_{45}A_{1,2,4}.
$$
The elements  $f_{1}f_{2}f_3$, $f_{13}f_3$, $f_{14}f_2$ and $f_{15}$
form a basis in the subspace of weight $\ve_5-\ve_1$ in $U_q(\g_-)$.
With $f_{12}=f_1$, $f_{23}=f_{45}=f_2$,  $f_{24}=f_{35}=f_{3}$, we rewrite $\check f_{15}=\hat f_{15}\bar A_1\ldots \bar A_4$ as
$$
\check  f_{15}=f_{15}\bar A_2\bar A_3 \bar A_4+f_{13}f_{3}\bar A_4\bar A_2+f_{14}f_{2}\bar A_3\bar A_2
+f_{1}f_{2}f_{3}(q-\bar q+A_{3}+A_{4})\bar A_3\bar A_4.
$$
Taking in to account that $(q-\bar q+A_{3}+A_{4})\bar A_3\bar A_4=\frac{y_3y_4-1}{q-\bar q}=\bar A_2$, we get the regularization
$$
\check  f_{15}/\dt_5^-=f_{15}\frac{y_3-1}{q-\bar q}\frac{y_4-1}{q-\bar q}+f_{13}f_{3}\frac{y_4-1}{q-\bar q} +f_{14}f_{2}\frac{y_3-1}{q-\bar q}
+f_{1}f_{2}f_{3},
$$
which never turns zero.
\end{example}

\section{Regularization of positive Mickelsson generators}

In this section, we regularize Mickelsson generators of  positive weights.
The assignment $f_\al\leftrightarrow e_\al$, $q^{\pm h_\al}\mapsto q^{\mp h_\al}$
 extends to an algebra automorphism $\omega\colon U_q(\g)\to U_q(\g)$.
 Denote $e_{ji}=\omega(f_{ij})\in U_q(\g_+)$.

Fix $j>1$.
Given $j\prec l\leqslant N$, define $D_l^j=\frac{q^{\eta_{l1}-\eta_{j1}}}{[\eta_{j1}-\eta_{l1}]_q}\in \hat U_q(\h)$. Set $D^j_{\vec m}=D^j_{m_1}\ldots D^j_{m_k}$ for
 a route $\vec m=(m_1,\ldots,m_k)$ with $j\prec m$.
Define
$$
\hat z_{\al_j}=e_{j1}+\sum_{j\prec   \vec m \prec k\leqslant N}^\varnothing f_{j,\vec m,k}e_{k1}D_{\vec m,k}^j(-1)^{|\!\!|j-k|\!\!|}q^{\eta_{j1}-\eta_{k1}} \in \hat U_q(\g),\quad  j=2,\ldots,N.
$$
and $\check z_{\al_j}=\hat z_{\al_j}\prod_{j\prec l} \bar D_l^{j}\in U_q(\g)$.
The elements $\hat z_{\al_j}$ with $\al_j=\ve_1-\ve_j\in \Rm^+$ form a set of positive Mickelsson generators, \cite{AM3}.

Introduce an ani-algebra automorphism $\tau$  of $U_q(\g_-)$ that is identical on the Chevalley generators. The Serre relations
imply that $\tau$ is well defined.
\begin{lemma}
For all $i,j$, $\tau(f_{ij})=f_{j'i'}$.
\end{lemma}
\begin{proof}
The proof can be conducted by considering  explicit expressions (\ref{a}-\ref{d_2}) for $f_{ij}$.
It is immediate for general linear and odd  orthogonal $\g$.
In other cases, the hardest part of the proof boils down  to checking
$[f_{jn},f_{ni}]_{\bar q^{1+\dt_{ij'}}}=[f_{jn'},f_{n'i}]_{\bar q^{1+\dt_{ij'}}}$ for $i<n$, $n'<j$, $\g=\s\p(2n)$ and $\g=\s\o(2n)$.
This can be done by applying the modified Jacobi identity
\be
[x,[y,z]_a]_b=[[x,y]_c,z]_{\frac{ab}{c}}+c [y,[x,z]_{\frac{b}{c}}]_{\frac{a}{c}},
\label{Jacobi}
\nn
\ee
for appropriate values of $a,b,c\in \C$. They are chosen so as to kill one of the internal commutators, making use of the
relations in the subalgebras of A-type in $U_q(\g)$.
\end{proof}

Introduce $\dt_j^-\in \A_j$ as follows. Put $\dt^-_j=1$ for $j<\srm'$ and, for $j\geqslant \srm'$,  $\dt_j^-=\frac{y_{j'}-1}{q-\bar q}=\bar A_{j'}$ if $\s\o(2n)$ and $\dt_j^-=\frac{y_{j'}^{\frac{1}{2}}+1}{q+1}$ if $\s\o(2n+1)$.

Denote
$\dt_j^+=\frac{q^{2\eta_{j1}-2\eta_{j'1}}-1}{q-\bar q}$ for $\s\o(2n)$,
and $\dt_j^+=\frac{qq^{2\eta_{j1}-2\eta_{*1}}+1}{q+1}$  for $\s\o(2n+1)$ assuming $j\leqslant s$. In all other cases including symplectic $\g$, put $\dt_j^+=1$.
\begin{thm}
For each $j=2,\ldots,N$ the element $\check z_{\al_j}$ is divisible by $\dt_j^+$. The quotient
$\check z_{\al_j}(\la)/\dt_j^+(\la)$ does not turn zero at all $\la \in \h^*$.
\end{thm}
\begin{proof}
We can assume that $\g$ is orthogonal and $s'\leqslant j'$.
For all $\la\in \h^*$,  define $\check z_{jk}\in U_q(\g_-)$, $k=j,\ldots, N$, setting
$\check z_{jj}=\prod_{j< l}\bar D^j_l$
and
$\check z_{jk}=
\sum_{j\prec \vec m \prec k}^\varnothing f_{j,\vec m,k}D_{\vec m,k}^j\check z_{jj}$
for $j\prec k$.
Then
$
\check z_{\al_j}(\la)=\sum_{k=j}^N(-1)^{|\!\!|j-k|\!\!|}q^{\eta_{j1}-\eta_{k1}}\check z_{jk}(\la) e_{k1}.
$
By the PBW theorem for $U_q(\g)$, $\check z_{\al_j}(\la)=0$ if and only if $\check z_{jk}(\la)=0$ for all $k$.

The assignment  $y_{k'}=q^{2\eta_{j1}-\eta_{k1}}$, $j\prec  k$, sends $\dt_{j'}^-$ over to $\dt_j^+$. 
One can check that the algebra generated by $\{y_l\}_{l<j'}$ is isomorphic to  $\A_{j'}$.
Indeed, $y_ky_{k'}=y_{j}$ for all  $k\not =*$ and $q^2y_*^2=y_{j}$ (odd $N$).

The anti-automorphism $\tau$ takes
$\check z_{jk}$
to $\check f_{k'j'}(\prod_{l}\bar A_l)$, where the product is done over those $l< j'$ which $k'\not \preccurlyeq l$; it includes
all $\bar A_l$ with $k'<l$.
By Lemma \ref{regularitB_so}, $\check f_{k'j'}$ is divisible by $\dt_{j'}^-$ if $k'\leqslant j$,  otherwise $\prod_{l}\bar A_l$
contains $\bar A_j$ divisible by $\dt_{j'}^-$.  This proves that $\check z_{\al_j}$ is divisible by $\dt_{j}^+$.
Finally, the element  $\check f_{1j'}/\dt_{j'}^-$ is the $\tau$-image of $\check z_{jN}/\dt_{j}^+$, up to the factor $(-1)^{|\!\!|j-k|\!\!|}y_{k'}
\not =0$.
Since $\check f_{1j'}/\dt_{j'}^-$ does not vanish at all weights, $\check z_{\al_j}/\dt_{j}^+\not =0$ at all weights too.
\end{proof}
\section{Application: decomposition of $V\tp M_\la$}
Define an ascending sequence of submodules $W_j=\sum_{i=1}^j M_{i}\subset V\tp M_\la$, for all $j=1,\ldots, N$.
Lemma \ref{principal} implies that $W_j\subset V_j$. We apply the regularization analysis to
answer the question when $V\tp M_\la$ is a direct sum of  $M_{j}$. Clearly that is the case if
the eigenvalues $x_j$ of $\Q$ are pairwise distinct. However the converse is not true.
\begin{propn}
\label{W=V}
For all $j\in [1,N]$, the following statements are equivalent:
{\em i)}  $V_j= W_j$,  {\em ii)}  $V_i=W_i$ for all $i\leqslant j$,
{\em iii)} projection $\wp_i\colon M_i\to V_i/V_{i-1}$ is an isomorphism  for all $i\leqslant j$,
 {\em iv)} $W_j=\op_{i=1}^j M_i$.
\end{propn}
\begin{proof}
Since all $M_{i}$ and $ V_i/V_{i-1}$ are Verma modules of the same highest weight,  iii) is equivalent to  $\wp_i$  being surjective or injective.
The implication ii) $\Rightarrow$ i) is trivial.
With $W_1=V_1$, assume that ii) is violated and let $k> 1$ be the smallest such that $W_k\not = V_k$.
Then $V_k$ and $W_k$ have different multiplicities of weight $\la+\ve_k$, so that  i) $\Rightarrow$ ii).

Assuming  ii) we find that all maps $\wp_i$ are surjective, hence  iii). Conversely, iii) implies that all maps $W_i\to V_i/V_{i-1}$ are surjective. Since, $W_1=V_1$,
 induction on
$i$ proves ii).

If $\wp_i$ are injective, then $M_i\cap W_{i-1}\subset M_i\cap V_{i-1} =\{0\}$,
that is, iii)  $\Rightarrow$  iv). Finally, suppose iv) and prove iii) by induction: assuming
iii) $\Leftrightarrow$ ii) true for $i<j$, the map $\wp_{j}$ is injective and hence an isomorphism. 
\end{proof}
Let $u_j= \check F_j/d^-_j\in V\tp M_\la$ be the regularized singular vector and introduce $C_j(\la)\in \C$ through the equality $\wp_j(u_j)=C_j(\la)v_{\la,j}$.
\begin{corollary}
  The module $V\tp M_\la$ is a direct sum $\op_{j=1}^N M_{j}$ if and only if all $C_j(\la)\not=0$.
\end{corollary}
\begin{proof}
  Since $u_j$ does not vanish at all $\la$, $C_j(\la)=0$ if and only if $\wp_j$ is an isomorphism.
\end{proof}
Denote $\phi_{ij}=\frac{q^{2\xi_{ij}}-1}{q-\bar q}$ for $i<j$ apart from $i=j'$ for $\g=\s\o(N)$, in which
case we set $\phi_{j'j}=1$ for even $N$ and $\varphi_{j'j}=\frac{q^{\xi_{j'j}}-1}{q-\bar q}$ for odd $N$.
Then $C_j\simeq \prod_{i=1}^{j-1}\phi_{ij}$  up to a numerical factor that never turns zero.
Recall that $q^{2\xi_{ij}}=x_ix_j^{-1}$, where $x_i$ are the eigenvalues of $\Q$.
If they are all distinct, then clearly $V\tp M_\la=\op_{j=1}^N M_{j}$. It follows that the converse is also true
for symplectic $\g$.
For orthogonal $\g$, the spectrum of $\Q$ does not separate the submodules $M_{j}$
at all weights. 
For instance, choose a weight $\la$ such that $q^{2(\la+\rho,\ve_j)}=-1$.
Then $x_jx_{j'}^{-1}=q^{4(\la+\rho,\ve_j)}=1$. If all
other $x_l$  are pairwise distinct and different from $x_j$, then the direct sum decomposition still holds.
This phenomenon facilitated quantization of borderline Levi conjugacy classes
 in \cite{AM3}.

\end{document}